\documentclass[9pt,technote]{IEEEtran}
\usepackage{algorithm}
\usepackage{algorithmic}
\usepackage{color}
\usepackage{float}
\usepackage{ifpdf}
\usepackage{graphicx}
\usepackage{amsfonts}
\usepackage{amsmath}
\usepackage{amssymb}
\usepackage{amsthm}
\usepackage{booktabs}

\newtheorem{proposition}{Proposition}
\newtheorem{theorem}{Theorem}

\newtheorem{lemma}{Lemma}

\newtheorem{remark}{Remark}
\newtheorem{prob}{Problem}

\title{Stochastic Sensor Scheduling for Energy Constrained Estimation in Multi-Hop Wireless Sensor Networks}

\author{Yilin Mo$^*$, Emanuele Garone$^\dag$, Alessandro Casavola$^\dag$, Bruno Sinopoli$^*$
\thanks{
$*$: Department of Electrical and Computer Engineering, Carnegie Mellon University, Pittsburgh, PA. Email: {ymo@andrew.cmu.edu, brunos@ece.cmu.edu}
}
\thanks{
$\dag$: Dipartimento di Elettronica, Informatica e Sistemistica, Universit\`{a} degli Studi della Calabria, Via Pietro Bucci, Cubo 42-c, Rende (CS), 87036, Italy. Email: {egarone, casavola@deis.unical.it}
}
\thanks{
This research was supported in part by CyLab at Carnegie Mellon under grant DAAD19-02-1-0389 from the Army Research Office. Foundation. The views and conclusions contained here are those of the authors and should not be interpreted as necessarily representing the official policies or endorsements, either express or implied, of ARO, CMU, or the U.S. Government or any of its agencies.}
}

\begin{document}
\maketitle
\begin{abstract}
Wireless Sensor Networks (WSNs) enable a wealth of new applications where remote estimation is essential. Individual sensors simultaneously sense a dynamic process and transmit measured information over a shared channel to a central fusion center. The fusion center computes an estimate of the process state by means of a Kalman filter. In this paper we assume that the WSN admits a tree topology with fusion center at the root. At each time step only a subset of sensors can be selected to transmit observations to the fusion center due to a limited energy budget. We propose a stochastic sensor selection algorithm that randomly selects a subset of sensors according to certain probability distribution, which is opportunely designed to minimize the asymptotic expected estimation error covariance matrix. We show that the optimal stochastic sensor selection problem can be relaxed into a convex optimization problem and thus solved efficiently. We also provide a possible implementation of our algorithm which does not introduce any communication overhead. The paper ends with some numerical examples that show the effectiveness of the proposed approach.
\end{abstract}
\section{Introduction}\label{sec:introduction}
Sensor networks span a wide range of applications, including environmental monitoring and control, health care, home and office automation and traffic control \cite{wireless_sensor_network}.  In these applications, estimation algorithms like Kalman filters can be used to undertake state estimation tasks based on lumped-parameter models of distributed physical phenomena. However, WSN operating constraints, such as power limitations, often make it difficult to collect data from every sensor at the sampling rates required for an effective monitoring. These considerations have led to the development of sensor scheduling strategies able to select, at each time step, the subset of reporting sensors that minimizes a certain cost function, usually related to the expected estimation error.

Sensor network energy consumption minimization and, consequently, lifetime maximization problems have been active areas of research over the past few years, as researchers realized that energy limitations constitute one of the major obstacles to the extensive adoption of such a technology. Sensor networks energy minimization is typically accomplished via efficient MAC protocols~\cite{mac_sensor_network} or via efficient scheduling of sensor states~\cite{wake_up_scheme, sleep_schedule}. In \cite{lifetime_analysis}, Xue and Ganz showed that the lifetime of sensor networks is influenced by transmission schemes, network density and transceiver parameters with different constraints on network mobility, position awareness and maximum transmission ranges.  Chamam and Pierre~\cite{optimal_state_scheduling} proposed a sensor scheduling scheme capable of optimally putting sensors in active or inactive modes. Shi et. al~\cite{Ling_cdc07} considered sensor energy minimization as a mean to maximize the network lifetime while guaranteeing a desired quality of the estimation accuracy. Moreover in~\cite{shi-febid08}, they proposed a sensor tree scheduling algorithm which leads to longer network lifetimes.

Conversely, optimizing the performance of sensor networks under given energy constraints, which can be seen as the dual problem of network energy minimization, has also been studied by several researchers. Such a constrained optimization problem has been studied for continuous-time linear systems in \cite{stochasticcontrol} and \cite{sensorprecision}. In \cite{hmmsensorselection}, the author computed the optimal sensor scheduling for the estimation of a Hidden Markov Model based system. For discrete-time linear systems, methods like dynamic programming \cite{dynmaicprogrammingsensorselection} or greedy algorithms \cite{oshmansensor} have been proposed to find the optimal sensor scheduling over long time horizons.

Another important contribution on the topic is the work of Joshi and Boyd \cite{Joshi}, where a general single-step sensor selection problem was formulated and solved by means of convex relaxation techniques. Such a paper provides a very general framework that can handle various performance criteria and energy and topology constraints. Following this work, Mo et al. \cite{Mo_ascc09, Mo_necsys09, Mo_allerton09} showed that multi-step sensor selection problems can also be relaxed into convex optimization problems and thus efficiently solved.

A very different approach with respect to the above deterministic solutions has been proposed in \cite{Gupta_sensorselection_automatica}. There, the authors proposed a stochastic sensor selection algorithm in networks endowed with star topology. The algorithm is based on the idea that at each time step the sensors randomly and autonomously choose if sending measurements or not according to a certain probability distribution. Therefore, the probability distributions become the optimization parameters, which are chosen to minimize the expected steady-state error covariance matrix. The authors argued that such a stochastic approach has several advantages over the conventional approaches: for example, it is easier to take into account random communication channel failures, which is a quite common issue in wireless sensor networks. The most relevant limitation of the results presented in that paper hinges upon the assumption that only one sensor at the time can transmit its data at each sampling period, which is a strong assumption and requires a precise coordination between sensors.

In the present work, we go further on by proposing a stochastic sensor selection algorithm that not only overcomes the above limitation but also solves the routing problem under the assumption that wireless sensor network has a tree topology. The proposed approach may be summarized as follows. The sensors are randomly selected according to a certain probability distribution that is designed so as to minimize the expected asymptotic estimation error covariance matrix while maintaining the connectivity of the network. In order to make the determination of the above probability distribution tractable, the problem is relaxed and, instead of the original objective function, a lower bound to the expected estimation error covariance matrix is minimized. Such a choice reduces the optimal sensor scheduling design problem into a convex optimization problem.
The advantages of the stochastic schedule over deterministic schedule can be summerized as threefold: 
\begin{enumerate}
\item The search space of the stochastic formulation is continuous and convex, while the search space of deterministic formulation is discrete. Hence, the search of the optimal deterministic schedule can be formulated as an integer programming problem, which makes the task potentially harder than the stochastic counterpart.
\item The expected performance of the stochastic formulation can be better than the deterministic one. Moreover, due to the ergodicity of the random Riccati equation, we can prove that under mild assumptions almost every sample path of the stochastic schedule is better than the deterministic one if the system runs long enough.
\item The stochastic schedule can be implemented with the same computation and communication cost as the deterministic one. 
\end{enumerate}

The rest of the paper is organized as follows. In Section~\ref{sec:problem} we describe our system and communication model and introduce the deterministic sensor and stochastic selection problems. We further present an ergodicity result on the performance of the stochastic sensor scheduling method to show that stochastic formulation could improve the performance. In Section~\ref{sec:main}, we relax the stochastic sensor selection algorithm to render it solvable and propose an possible implementation of our algorithm. Some numerical examples on the monitoring of a diffusion process are provided in Section~\ref{sec:simulation} and, finally, Section~\ref{sec:conclusion} concludes the paper.

\section{Sensor Selection: From Deterministic to Stochastic Formulation}\label{sec:problem}
\subsection{System Description}
Consider the following discrete-time LTI system
\begin{equation}
    x_{k+1} = Ax_k + w_k
    \label{eq:systemdescription}
\end{equation}
where $x_k \in \mathbb R^n$ represents the state and $w_k\in \mathbb R^n$ the disturbance. It is assumed that $w_k$ and $x_0$ are independent Gaussian random vectors, $x_0 \sim \mathcal N(0,\;\Sigma)$ and $w_k \sim \mathcal N(0,\;Q)$, where $\Sigma,\,Q > 0$ are positive definite matrices. A wireless sensor network composed of $m$ sensing devices $s_1,\ldots,s_m$ and one fusion center $s_0$ is used to monitor the state of system \eqref{eq:systemdescription}. The measurement equation is 
\begin{equation}
  y_{k} = C x_k + v_k,
\end{equation}
where $y_k = [y_{k,1}' , y_{k,2}' ,\ldots, y_{k,m}']' \in \mathbb R^m$ is the measurement vector\footnote{The $'$ on a matrix always means transpose.}. Each element $y_{k,i}$ represents the measurement of sensor $i$ at time $k$. $C = \left[C_1',\ldots,C_m'\right]'$ is the observation matrix and the matrix pair $(C,\,A)$ is assumed observable\footnote{The assumption of observability is without loss of generality since we could perform Kalman decomposition and only consider the observable space even if the system is not observable.}. $v_k \sim \mathcal N(0,\;R)$ is the measurement noise, assumed to be independent of $x_0$ and $w_k$. We also assume that the covariance matrix $R = diag(r_1,\ldots,r_m) $ is diagonal, which means that the measurement noise at each sensor is independent of all others and nonsingular, that is $r_i > 0, i=1,...,m$.

Let's introduce an oriented communication graph $G = \{V\,,E\}$ in order to model the communication amongst nodes, where the vertex set $V = \{s_0,\,s_1,\ldots,s_m\}$ contains all sensor nodes, including the fusion center. The set of edges $E \subseteq V \times V$ represents the available connections, i.e.  $(i,j) \in E$ implies that the node $s_i$ may send information to the node $s_j.$ Moreover, it is assumed that each node of the sensor network acts as a gateway for a specific number of other nodes, which means that every time it communicates with another node it sends, in a single packet, its own measurements collected together with all data received from the other nodes.

We always assume that, for every sensor in the network, there exists one and only one communication path to the fusion center, i.e. the sensor network has a directed tree topology. Moreover, we assume that each link has an associated weight $c(e_{i,j})$ which indicates the energy consumed when $s_i$ directly transmits a packet to $s_j$. For the sake of legibility, we sometimes abbreviate $c(e_{i,j})$ as $c_i$, $i= 1,\ldots,m$ because, in the assumed topology, each sensor node has only one outgoing edge.
\begin{remark}
  The tree topology assumption may be a restrictive hypothesis in the general case where usually one sensor can communicate with several nearby nodes. However, it is worth to remark that typical communication network graphs can be approximated by a collection of ``representative" spanning trees (e.g. the first $m$ spanning trees of the spanning tree enumeration \cite{EnumerateSpanningTrees}).
\end{remark}
\subsection{Stochastic v.s. Deterministic Sensor Selection}  
Because sensor measurements usually contain redundant information, in order to reduce the energy consumption it would be highly desirable to use a minimal subset of sensors at each sampling time. However, in a tree topology, we cannot select arbitrary subsets of nodes but we are forced to select nodes (and connections) such that, for each selected node, there exists a communication path to the fusion node. As a result, any possible transmission topology of $G$ is a subtree $T = \{V_T\,,E_T\}$, with $s_0 \in V_T$,  $V_T\subseteq V$ and $E_T\subseteq E$. Hereafter, $V_T$ denote the selected subset of sensors and $E_T$ the communication paths used by the sensors to transmit observations to the fusion center. We also denote by $\mathcal T$ the set of all possible transmission topologies $T$ (i.e. the set of all possible subtrees of $G$ containing $s_0$).

It is straightforward to show that, for a transmission tree $T$, the total transmission energy consumption is given by\footnote{Here we assume that $cost(e_{i,j})$ is constant regardless of number of observations contained in the packet. This is realistic in most of the cases, especially when measurements are of simple type, such as low precision scalar values, and the transmission overhead, e.g. header, handshaking protocol, dominates the payload.}
\begin{displaymath}
  \mathcal E(T) = \sum_{e \in E_T} c(e).    
\end{displaymath}

Suppose that at each time $k$ we randomly select a tree $T$ from $\mathcal T$  and each sensor in $T$ transmits its observation back to the fusion node according to the topology $T$. Let $\pi_{k,T}$ be the probability that the transmission tree $T$ is selected at time $k$. Then, we may define
\begin{equation}
  p_{k,i} \triangleq \sum_{T\in \mathcal T, s_i \in V_T} \pi_{k,T}
\end{equation}
the marginal probability that sensor $i$ is selected at time $k$. Further, let us define $\mathbf p_k = [p_{k,1},\ldots,p_{k,m}]'$ and $\mathbf \pi_k=[\pi_{k,T_1},\ldots,\pi_{k,T_{|\mathcal T|}}]'$ to be the vectors of all $p_{k,i}$s and $\pi_{k,T}$s respectively. We can introduce the binary random variable $\delta_{k,T}$ such that $\delta_{k,T}=1$ if the transmission tree $T$ is selected at time $k$ and $\delta_{k,T}=0$ otherwise. Similarly, let us also define the binary random variable $\gamma_{k,i}$ to be $1$ if sensor $i$ is selected at time $k$ and $0$ otherwise. It is well known that the Kalman filter is still the optimal filter\cite{Gupta_sensorselection_automatica}. Suppose that $V_T = \{s_0,\,s_{i_1},\ldots,s_{i_j}\}$, then we can define
\begin{eqnarray}
  C_{T} & \triangleq & [C_{i_1}',\,C_{i_2}',\ldots,C_{i_j}']',  R_{T}  \triangleq   diag(r_{i_1},\ldots,r_{i_j}).
\end{eqnarray}
It can be proved that the estimation error covariance $P_k$ and the information matrix $Z_k$\footnote{The information matrix is the inverse of estimation error covariance} of the Kalman filter satisfy the following recursive equations:
\begin{align}
  P_k &= \left(P_{k|k-1}^{-1} + C_T'R_T^{-1}C_T\right)^{-1},
\label{eq:informationmatrix}
\end{align}
where $ P_{k|k-1}= AP_{k-1}A' + Q.$
Let us define $\mathbf g_{\mathbf \pi_k,k}$ as a random operator such that
\begin{equation}
  \mathbf g_{\mathbf \pi_k,k}(X) \triangleq \sum_{T\in\mathcal T}\delta_{k,T} g_T(X),
\end{equation}
where $P(\delta_{k,T} = 1) = \pi_{k,T}$, and
\begin{equation}
  g_T(X) \triangleq \left[(AXA'+Q)^{-1} + \sum_{s_i \in V_{T},\, s_i\neq s_0} \frac{C_iC_i'}{r_i}\right]^{-1}.
\end{equation}
We have
\begin{equation}
  P_k = \mathbf g_{\mathbf \pi_k,k}(P_{k-1}).
\end{equation}
In this paper we are more interested in a time-invariant schedule $\pi_T$. Hence, let us define 
\begin{equation}
  \mathbf g^{\infty}_\pi(X)\triangleq \lim_{k\rightarrow\infty}\mathbb E (g_{\pi,k}\circ g_{\pi,k-1} \circ\cdots\circ g_{\pi,1})(X),
\end{equation}
when the limit exists. Otherwise, $\mathbf g^{\infty}_\pi(X)$ is infinity. Note that $\mathbf g^{\infty}_\pi$ is a deterministic function, which indicates the limit performance of stochastic sensor selection when the fixed schedule $\pi$ is used. It is easy to see that
\begin{displaymath}
  \lim_{k\rightarrow\infty} \mathbb EP_k = \mathbf g^{\infty}_\pi(\Sigma),
\end{displaymath}
when the fixed schedule $\pi$ is used and $ \mathbf g^{\infty}_\pi(\Sigma)<\infty$.

Since transmission trees are randomly selected, $P_k$ is a random matrix. Thus, we only minimize the asymptotic expected estimation error covariance matrix while requiring that the expected energy consumption does not exceed a designated threshold $\mathcal E_d$. The problem of finding the optimal fixed stochastic schedule that minimizes the expected asymptotic estimation error covariance matrix can be formulated as
\begin{prob}[Fixed Random Schedule that Optimizes Expected Asymptotic Performance]\label{randoptasymptotic}
  \begin{align*}
    &\mathop{\textrm{minimize}}\limits_{\mathbf \pi }&
    &trace(\mathbf g_\pi^\infty(\Sigma))\\
 &\textrm{subject to}&
 &   \sum_{T\in \mathcal T} \pi_{T}\mathcal E(T)\leq\mathcal E_d,\, \pi_{T}\geq 0,\, \sum_{T\in\mathcal T} \pi_{T} = 1.
  \end{align*}
\end{prob}

Since the deterministic schedule can be seen as a subset of stochastic schedule, where $\pi_{k,T}$ are forced to be either $0$ or $1$, the problem of finding the optimal fixed deterministic schedule that minimizes the asymptotic estimation error covariance matrix can be formulated as
\begin{prob}[Fixed Deterministic Schedule that Optimizes Asymptotic Performance]\label{detoptasymptotic}
  \begin{align*}
    &\mathop{\textrm{minimize}}\limits_{\mathbf \pi }&
    &trace(\mathbf g_\pi^\infty(\Sigma))\\
 &\textrm{subject to}&
 &   \sum_{T\in \mathcal T} \pi_{T}\mathcal E(T)\leq\mathcal E_d,\, \pi_{T}=0\;or\;1,\, \sum_{T\in\mathcal T} \pi_{T} = 1.
  \end{align*}
\end{prob}

\begin{remark}
  In Problem~\ref{randoptasymptotic} we require that the expected energy consumption does not exceed a certain energy budget. In real applications different constraints may be considered (e.g. requirements on the sensor lifetime). However, it can be shown (see e.g. \cite{Joshi}) that many of these constraints can be easily integrated into the above framework.
\end{remark}

\begin{remark}
It is worth noticing that at each sampling time, the energy cost of deterministic schedule cannot exceed the designated threshold $\mathcal E_d$. This is important to be remarked in order to understand why stochastic sensor selections, being allowed to use more energy at one single sampling period, can achieve better performance than the above deterministic formulation.

It is also worth noticing that a periodic schedule can also be formulated as Problem~\ref{detoptasymptotic} by enlarging the state space. As a result, all the results in this section can be generalized in to periodic schedule. However, in Section~\ref{sec:main} we focus only on time-invariant schedule.
\end{remark}
%

\begin{remark}
  \label{remark:stochasticvsdeterministic}
  Another main difference between Problem~\ref{randoptasymptotic} and Problem~\ref{detoptasymptotic} is that, the search space of deterministic schedule is discrete, which that of stochastic schedule is continuous and convex. This brings several advantages. First, the deterministic schedule can be seen as a particular kind of random schedule, where $\pi_{k,T}$s are binary. As a result, stochastic sensor selection strategies could possibly improve the sensor selection performance (at least in the expected sense). The second advantage is that the feasible set $\pi_{k,T}$ is convex, which allows us to further manipulate the problem into a convex form.
\end{remark}

As is commented above, the expected performance of the optimal stochastic schedule is better than the deterministic counter part. Let $\pi^*$ be the optimal stochastic schedule and $\pi^*_d$ be the optimal deterministic schedule, we have
\begin{displaymath}
  \lim_{k\rightarrow\infty} \mathbb E\, trace(P_k(\pi^*))  \leq \lim_{k\rightarrow\infty}  trace(P_k(\pi_d^*)),
\end{displaymath}
which implies that
\begin{displaymath}
  \lim_{N\rightarrow\infty} \sum_{k=1}^N \frac{1}{N}\mathbb E\, (traceP_k(\pi^*))  \leq  \lim_{N\rightarrow\infty} \sum_{k=1}^N \frac{1}{N} traceP_k(\pi_d^*).
\end{displaymath}

To strength this result, the following theorem states that if the optimal stochastic schedule is allowed to run for a long time, then almost every sample path of the stochastic schedule is potentially better than deterministic one in the average sense.
\begin{theorem}
  Suppose that the fixed schedule $\mathbf \pi^*$ is the solution of Problem~\ref{randoptasymptotic}. If the linear system and $\mathbf \pi^*$ satisfy the following assumptions:
  \begin{enumerate}
    \item $A$ is invertible, $(A,Q^{1/2})$ is controllable;
    \item there exists a transmission topology $T$ with $\pi^*_{T} > 0$ such that  $(C_T,\,A)$ is observable
  \end{enumerate}
and the stochastic process $\{P_k\}$ satisfies: $P_k = \mathbf g_{\pi^*,k} (P_{k-1}),\, P_0 = \Sigma,$ then almost surely the following inequality holds
  \begin{equation}
    \lim_{N\rightarrow \infty} \frac{1}{N}\sum_{k=1}^N trace(P_{k})  \leq trace(\mathbf g_{\pi^*}^\infty(\Sigma)).
    \label{eq:ergodic}
  \end{equation}
  \label{theorem:ergodic}
\end{theorem}
\begin{proof}
  It is easy to check that all the assumptions in the Theorem 3.4 of \cite{randomriccati} hold. As a result, there exists an ergodic stationary process $\{\overline P_k\}$ which satisfies $\overline P_k = \mathbf g_{ {\pi^*},k}(\overline P_{k-1})$. Moreover,
  \begin{displaymath}
    \lim_{k\rightarrow \infty} \|P_k -\overline P_k\| = 0.\,a.s.
  \end{displaymath}
  We want to prove that $\mathbb E(trace(\overline P_0))$ is less than or equal to $trace(\mathbf g_\pi^\infty(\Sigma))$ and hence is finite. Because $\overline P_k$ is ergodic, and $P_k$ converges to $\overline P_k$ almost surely, we know that
  \begin{align*}
    &\lim_{N\rightarrow\infty}\frac{1}{N}\sum_{k=1}^N \min(trace(P_k),M) =  \lim_{N\rightarrow\infty}\frac{1}{N}\sum_{k=1}^N \min(trace(\overline P_k),M)\\
    &=\mathbb E[\min(trace(\overline P_0),M)],\;\;\;a.s.
  \end{align*}
  where $M > 0$ is a constant. By the definition of $\mathbf g_\pi^\infty$, we know that
  \begin{displaymath}\begin{split}
   & trace(\mathbf g_\pi^{\infty}(\Sigma))  \geq \lim_{N\rightarrow\infty} \mathbb E\left[\frac{1}{N} \sum_{k=1}^N \min(trace(P_k),M)\right]\\
    & = \mathbb E\left[ \lim_{N\rightarrow\infty} \frac{1}{N} \sum_{k=1}^N \min(trace(P_k),M)\right] = \mathbb E[\min(trace(\overline P_0),M)].
  \end{split}
  \end{displaymath}
The second equality follows from the Dominated Convergence Theorem. Now, let $M\rightarrow \infty$. By Monotone Convergence Theorem it results that
\begin{displaymath}
  \mathbb E[trace(\overline P_0)] = \lim_{M\rightarrow\infty}\mathbb E[\min(trace(\overline P_0),M)]\leq trace(\mathbf g_\pi^{\infty}(\Sigma)),
\end{displaymath}
which proves that $\mathbb E[trace(\overline P_0)] \leq trace(\mathbf g_\pi^\infty(\Sigma))$. Hence, by ergodicity, we obtain
  \begin{align*}
    &\lim_{N\rightarrow \infty} \frac{1}{N}\sum_{k=1}^N trace(P_{k}) =  \lim_{N\rightarrow \infty} \frac{1}{N}\sum_{k=1}^N trace(\overline  P_{k}) = \mathbb E(trace(\overline P_0))\\
    &\leq trace(\mathbf g_\pi^\infty(\Sigma)),\;\;\;a.s.
  \end{align*}
\end{proof}

\begin{remark}
  Combining Remark~\ref{remark:stochasticvsdeterministic} with the results of Theorem~\ref{theorem:ergodic}, we can conclude that the average performance of almost every sample path of the optimal stochastic schedule is better than its deterministic counterpart. 
\end{remark}

Before moving forward, it is worth pointing out that Problem~\ref{randoptasymptotic} are still numerical intractable. In fact:
\begin{enumerate}
  \item it is usually difficult to express $\mathbb EP_{\infty}$ as an explicit function of $\pi_{1,T},\,\ldots,\,\pi_{k,T}$;\footnote{The readers can refer to \cite{sinopoli} for more information.}
  \item since $|\mathcal T|$ is large, the number of optimization variables and constraints may be not polynomial with respect to the number of nodes. 
\end{enumerate}
In the next section, we will devise a possible relaxation method that allows one to overcome the above two problems.

\section{Relaxation and Implementation}\label{sec:main}
In this section, we first relax Problem~\ref{randoptasymptotic} to a convex relaxation problem. We then propose a possible implementation of our stochastic schedule without introducing communication and computation overhead.
\subsection{Relaxation}
In this subsection we consider a convex relaxation of Problem~\ref{randoptasymptotic}. To this end, let us define a lower bound $L_k$ to $\mathbb EP_k$ by means of the following theorem, whose proof is reported in the Appendix.
\begin{theorem}\label{theorem:lowerbound}
  Let $L_0 =  P_0$ and
  \begin{equation}
	  L_k = \left(L_{k|k-1} ^{-1} + \sum_{i=1}^m p_{k,i} \frac{C_iC_i'}{r_i}\right)^{-1},
	  \label{eq:lowerbound1} 
  \end{equation}
  where $L_{k|k-1} = AL_{k-1}A' + Q.$ The following inequalities hold: 
  \begin{equation}
    \mathbb EP_k \geq L_k. 
  \end{equation}
\end{theorem}

\noindent To further improve the legibility, let us define the function 
\begin{equation}
  L(X,\mathbf p) \triangleq \left[(AXA' + Q)^{-1} + \sum_{i=1}^m p_{i} \frac{C_iC_i'}{r_i}\right]^{-1}, 
\end{equation}
where $X \in \mathbb R^{n\times n}$ is positive semidefinite and $\mathbf p = [p_1,\ldots,p_m]'\in \mathbb R^m$. Moreover, let us define,
\begin{equation}
 L^{(1)}(X,\mathbf p) = L(X,\mathbf p), L^{(k)} (X,\mathbf p) = L(L^{(k-1)}(X,\mathbf p),\mathbf p),
\end{equation}
with 
\begin{equation}
  L^{\infty}(X,\mathbf p) = \lim_{k\rightarrow\infty} L^{(k)}(X,\mathbf p),
  \label{eq:asymptoticlowerbound} 
\end{equation}
when the limit exists. Hence \eqref{eq:lowerbound1} can be simplified as 
\begin{equation}
  L_k = L(L_{k-1}, \mathbf p_k). 
  \label{eq:lowerbound3}
\end{equation}
By replacing the objective function in Problem~\ref{randoptasymptotic} with its lower bound, we obtain the following:
\begin{prob}[Asymptotic Lower Bound for Random Transmission Tree Selection]\label{randoptasymptoticlower}
\begin{align*}
  &\mathop{\textrm{minimize}}\limits_{\pi_{T},\,\mathbf p}&
  &trace(L^\infty(\Sigma,\mathbf p)) \\
  &\textrm{subject to}&
  &\sum_{T\in \mathcal T} \pi_{T}\mathcal E(T)\leq\mathcal E_d, \\
  &&&\pi_{T}\geq 0,\,\sum_{T\in\mathcal T} \pi_{T} = 1,\,p_{i} = \sum_{s_i \in V_T} \pi_{T}.
\end{align*}
\label{randoptlower}
\end{prob}

There are drawbacks of the above formulation: 1) the optimization problem still has a number of constraints and variables depending on $|\mathcal T|$, a number which is not, in the general case, polynomial with respect to $m$; 2) $L^\infty$ is still not explicity. Let us first drop the dependence on $\pi_{T}$. To this end, define the set of feasible $\mathbf p$ for Problem \ref{randoptlower}:
\begin{displaymath}
  \begin{split}
  &\mathcal P \triangleq \left\{\mathbf p\left|\exists \pi,\, \sum_{T\in \mathcal T} \pi_{T}\mathcal E(T)\leq\mathcal E_d,\, \pi_{T}\geq 0,\,\sum_{T\in\mathcal T} \pi_{T} = 1,\, p_{i} = \sum_{s_i \in V_T} \pi_{T}\right.\right\}.
  \end{split}
\end{displaymath}
The following results can be easily proved:

\begin{proposition}
  The energy cost of a given collection of tree selection probabilities $\pi_{k,T}, \forall T \in {\mathcal{T}}$ is a linear function of the resulting marginal probability: 
  \begin{equation}
    \sum_{T \in \mathcal T} \pi_{T} \mathcal E(T)= \sum_{i = 1}^m c_i p_{i}. 
  \end{equation}
  \label{proposition:energy}
\end{proposition}

\begin{proposition}
  \label{proposition:existence}
  If $p_{i}\in [0,1]$ and if it satisfies 
  \begin{equation}
    p_{i} \leq  p_{j},\qquad \textrm{if $j$ is a parent of $i$}
  \label{eq:descritpion1} 
  \end{equation}
  then there exists at least one collection of tree selection probabilities $\pi$, such that 
  \begin{equation}
  \pi_{T}\geq 0,\; \sum_{T\in\mathcal T} \pi_{T} = 1,\; p_{i} = \sum_{s_i \in V_T} \pi_{T}.
  \label{eq:descritpion2} 
  \end{equation}
  Conversely, if there exists  $\pi_k$ such that \eqref{eq:descritpion2} holds, then  $p_{k,i}\in [0,1]$ and satisfies \eqref{eq:descritpion1} .
\end{proposition}

\noindent By exploiting the above Propositions we can reformulate the feasible set $\mathcal P$ as follows 
\begin{equation}
  \mathcal P =\left\{\mathbf p\left|p_i\in[0,\,1],\,\sum_{i = 1}^m c_i p_{i}\leq \mathcal E_d,\,  p_{i} \leq  p_{j},\textrm{if $j$ is parent of $i$}\right. \right\},
  \label{eq:treecondition} 
\end{equation}
and we can rewrite Problem~\ref{randoptasymptoticlower} as
\begin{prob}[Asymptotic Lower Bound for Random Transmission Tree Selection]\label{randoptasymptoticlowerfinal}
\begin{align*}
  &\mathop{\textrm{mininize}}\limits_{\mathbf p \in \mathbb R^m}&
  &trace(L^\infty(\Sigma,\mathbf p)) \\
  &\textrm{subject to}&
  &\mathbf p \in \mathcal P.
\end{align*}
\end{prob}
Now the main difficulty to solve the above problem is that $L^\infty(X,\mathbf p)$ is in general not convex in $\mathbf p$. Moreover, the exact form of $L^\infty(X,\mathbf p)$ is unknown. To overcome those limitations, we propose the following algorithm:
\begin{enumerate}
  \item Define
     $ \mathbf p_0 =  \left({\mathcal E_d}/{\left(\sum_{i=1}^m c_i\right)}\right) \mathbf 1_m,
    $
    where $\mathbf 1_m \in \mathbb R^m$ is a vector with all one entries and choose the matrix
    $  L_0 = L^{\infty}(I_n,\mathbf p_0).  $ 
  \item Let $L_k$ and $\mathbf p_k$ be the solution of the following optimization problem
\begin{prob}[Random Sensor Selection with Descend Constraint]\label{greedylower}
\begin{align*}
  &\mathop{\textrm{minimize}}\limits_{\mathbf p_k \in \mathbb R^m}&
  &trace(L_k)(=trace(L(L_{k-1},\mathbf p_k)))\\
  &\textrm{subject to}&
  & L_k \leq L_{k-1},\,\mathbf p_k \in \mathcal P.
\end{align*}
\end{prob} 
  \item Choose $\mathbf p^*$ as an accumulation point of $\mathbf p_k$\footnote{An accumulation point of a sequence is the limit of a converging subsequence}. Then $L^{\infty}(X, \mathbf p^*) = \lim_{k\rightarrow \infty} L_k$ for any $X\geq 0$. 
\end{enumerate}
Before proving the feasibility of the above algorithm, we want to point out that our algorithm is greedy. In fact, we try to minimize the lower bound for the next step in the hope of reducing the final asymptotic lower bound. As a result, it is suboptimal by nature. The following theorem gives a characterization of the main features of the proposed algorithm.

\begin{theorem}
  \label{theorem:lfunction}
$L(X,\mathbf p)$ is convex with respect to $\mathbf p$ and it is concave and monotonically increasing with respect to $X$.
\end{theorem}

\noindent  Due to the convexity of $L$ and $\mathcal P$, Problem~\ref{greedylower} is a convex optimization problem with $O(m)$ optimization variables and  $O(m)$ constraints. Thus, it can be solved efficiently. For example, if interior-points methods is used, then the complexity is $O(m^3)$. For detailed discussions about the computational burdens, please refer to \cite{Joshi}.

\begin{theorem}
The following statements are true for the proposed algorithm:
\begin{enumerate}
  \item $L_0$ exists. 
  \item Problem~\ref{greedylower} is always feasible. 
  \item $\mathbf p^*$ exists and $\mathbf p^*\in\mathcal P$. 
\item $L_\infty = \lim_{k\rightarrow \infty}L_k$ exists. 
\item $L_\infty = L^{\infty}(X,\mathbf p^*)$ for all positive semidefinite $X$. 
\end{enumerate}
\label{theorem:algorithm}
\end{theorem}

\begin{proof}
  \begin{enumerate}
    \item The proof is reported in the Appendix.
    \item Suppose that the Problem~\ref{greedylower} is feasible up to time $k$. To prove the problem is also feasible at time $k+1$, we only need to find one $\mathbf p\in \mathcal P$ and $L(L_k,\mathbf p)\leq L_k$. If we choose $\mathbf p = \mathbf p_k$ then, becasue $\mathbf p_k$ is the solution at time $k$, it follows that $\mathbf p_k\in\mathcal P$. It remains to prove that $L(L_k,\mathbf p_k)\leq L_k$, which can be proved by noticing that $L_k = L(L_{k-1},\mathbf p_k)  \leq L_{k-1}$ and $L(X,\mathbf p)$ is monotonically increasing with respect to $X$.
      Similarly, Problem~\ref{greedylower} is also feasible at time $1$ and then, by induction, Problem~\ref{greedylower} is always feasible. 
\item It is easy to see that $\mathbf p_k$ is bounded because $p_{k,i}\in [0,1]$. By means of the Bolzano-Weierstrass Theorem, this implies that there always exists an accumulation point $\mathbf p^*$. Moreover, because $\mathbf p_k\in\mathcal P$ and $\mathcal P$ is closed, $\mathbf p^*\in \mathcal P$. 
\item Because $\{L_k\}$ is decreasing and $L_k \geq 0$ for all $k$, the limit must exist. 
    \item The proof is reported in the Appendix.
  \end{enumerate}
\end{proof}

\begin{remark}
	It is worth noticing that in general it may exist more than one set of $\pi_{T}, \forall T \in \mathcal{T} $ with the same marginal probabilities. One possible way to determine $\pi_{T}$ is as follows:
	\begin{enumerate}
		\item Sort the marginal probability $p_{i}$, suppose that $p_{i_1}\geq p_{i_2}\geq\ldots\geq p_{i_m}$.
		\item Define $T_0 = \{s_0\}$, $T_j=T_{j-1}\bigcup \{i_j\}$.
		\item Choose $\pi_{T_0} = 1-p_{i_1}$, $\pi_{T_1}=p_{i_1}-p_{i_2},\pi_{T_2}=p_{i_2}-p_{i_3},\ldots,\pi_{T_m}=p_{i_m}$.
	\end{enumerate}
	One can easily verify that $T_i\in \mathcal T$ and $\pi_{T}$ are compatible with the marginal probability.
\end{remark}

\subsection{Implementation}
In this subsection we discuss a possible implementation of our sensor selection algorithm. We assume that a fixed random schedule $\mathbf p$ is used. Since the optimization does not depend on the real-time sensor measurement $y_k$, the optimization step is performed off-line in a centralized fashion. Each sensor $i$ stores its optimal $p_{i}$ and $p_{j}$ of all its children.

At each time $k$, we have to select one subset of sensors according to the marginal probabilities $\mathbf p$. However, we do not want the fusion center to query the nodes because this would increase the communication overhead, defying the purpose of sensor selection. To overcome this problem, we propose the following algorithm:
\begin{enumerate}
\item Every sensor is equipped with the same random number generator and the same seed. 
\item At time $k$, each sensor draws a random number $\alpha_k$ from the random number generator. 
\item If sensor $i$ has no children, then it compares $\alpha_k$ with $p_{i}$. If $\alpha_k \leq p_{i}$, then it transmits the measurement to its parent. Otherwise, it does not transmit anything. 
\item If sensor $i$ has children, then it compares $\alpha_k$ with $p_{j}$, where $j$ is the index of its child node. If $\alpha_k \leq p_{j}$, then sensor $i$ knows that child $j$ will forward an observation packet to him. After the node $i$ receives all the observation packets from its children, it merges all packets and its own observations into a single packet and forwards it to its parent. If $\alpha_k > p_{j}$ for all $j$ child of $i$, then the node $i$ compares $\alpha_k$ with $p_{i}$. If $\alpha_k \leq p_{i}$, then sensor $i$ transmits its measurements to its parent. Otherwise, it does not transmit anything. 
\end{enumerate}
\noindent Because all sensors are equipped with the same random number generator and the same seed, every sensor gets the same $\alpha_k$ at time $k$. Hence, the above algorithm guarantees that all sensors agree on the same transmission topology $T$ which satisfies the marginal distribution $\mathbf p$. It is worth to remark that in such a scheme the only communication needed is the transmission of the observation packets and no communication overhead for coordination purposes is needed.
\begin{remark}
It is worth mentioning that since all the sensors agree on the same $\alpha_k$, it is very easy to implement a Time Division Multiple Access (TDMA) protocol to avoid wireless interference. 	
\end{remark}

\section{Simulation Result}\label{sec:simulation}
In order to show the effectiveness of the proposed method we apply our stochastic sensor selection algorithm to a numerical example in which a sensor network is deployed to monitor a diffusion process in a $l\times l$ planar closed region, whose model is given by 
\begin{equation}
  u_t = \alpha \triangledown^2 u.
	\label{eq:diffussionpdfcont}
\end{equation}
where $\triangledown^2$ is the Laplace operator. $u(t,x_1,x_2)$ denotes the temperature at time $t$ at location $(x_1,x_2)$ and $\alpha$ indicates the speed of the diffusion process.

We use the finite difference method to discretize this model by dividing the region into $1m\times 1m$ grids and time into $1s$ slot. If we group all temperature values at time $k$ in the vector $U_k = [u(k,0,0),\ldots,u(k,0,N-1),u(k,1,0),\ldots,u(k,N-1,N-1)]^T$, we can write the evolution of the discretized system as $U_{k+1} = A U_{k}$, where the $A$ matrix can be computed from discretization. If we introduce process noise, $U_k$ will evolve according to $U_{k+1} = A U_{k} + w_k,$ where $w_k \in \mathcal N(0,\;Q)$ is the process noise.

We suppose that the fusion center is located in the bottom left corner at position $(0,0)$. We assume that $m$ sensors are randomly distributed in the region and each sensor measures a linear combination of temperature of the grid around it\footnote{We do not require the sensors to be placed at grid points}. In particular, if we suppose the location of sensor $l$ of coordinates  $(a_1,a_2)$ is in the cell $[i,j]$, i.e. $a_1 \in [i,i+1)$ and $a_2 \in [j,j+1)$, the measurement of this sensor is
\begin{displaymath}
	\begin{split}
	 &y_{k,l} = \left[\right.(1-\Delta a_1)(1-\Delta a_2) u(k,i,j)+  \Delta a_1(1-\Delta a_2) u(k,i+1,j)+\\
	 &(1-\Delta a_1)\Delta a_2 u(k,i,j+1)+  \Delta a_1\Delta a_2 u(k,i+1,j+1)\left.\right]/h^2+ v_{k,l}  .
	\end{split}
\end{displaymath}
where $\Delta a_1 =a_1-i$, $\Delta a_2 =a_2-j$ and $v_{k,l}$ is the measurement noise of sensor $l$ at time $k$. Indicating with $Y_k$ the vector of all the measurements at time $k$, it follows that: $Y_k =  C U_k + v_k,$ where $v_k$ denotes the measurement noise at time $k$ assumed to have normal distribution $\mathcal N(0,\;R)$ and $C $ is the observation matrix. Finally, we assume that the sensor network admits a minimum spanning tree topology with communication cost from sensor $i$ to $j$ is 
\[
cost(e_{i,j}) = c+d_{i,j}^2 
\]
where $d_{ij}$ is the Euclidean distance from sensor $i$ to sensor $j$ and $c$ is a constant related to the sensing energy consumption\footnote{$c$ models the fact that as the distance goes to zero the communication cost does not}. For the simulations, we impose the following parameters: $l=3\; m$, $m = 16$, $\alpha = 0.1 \; m^2/s$, $Q = I =R = I \in \mathbb R^{16\times16}$, $\Sigma = 4I \in \mathbb R^{16\times 16}$, $\mathcal E_d = 6$,$c=1$. 


We compare the performance of the optimal fixed stochastic schedule with optimal fixed deterministic schedule found by exhaustive search. Figure~\ref{fig:fixed} shows the histogram of the ratio between $trace(P_\infty)$ of deterministic schedule and $trace(EP_\infty)$ of stochastic schedule, which is generated by 100 random experiments. The blue dashed line is the average ratio. It can be seen that the deterministic schedule is always worse than the stochastic one. Figure~\ref{fig:fixed} shows the trace of $P_k$ for the optimal deterministic fixed schedule, together with the trace of $P_k$ from a sample path of the stochastic fixed schedule and the $EP_k$ of the stochastic fixed schedule for one random experiment.

\begin{figure}[<+htpb+>]
  \centering
  \begin{minipage}[t]{0.4\textwidth}
    \begin{center}
      \includegraphics[width=0.9\textwidth]{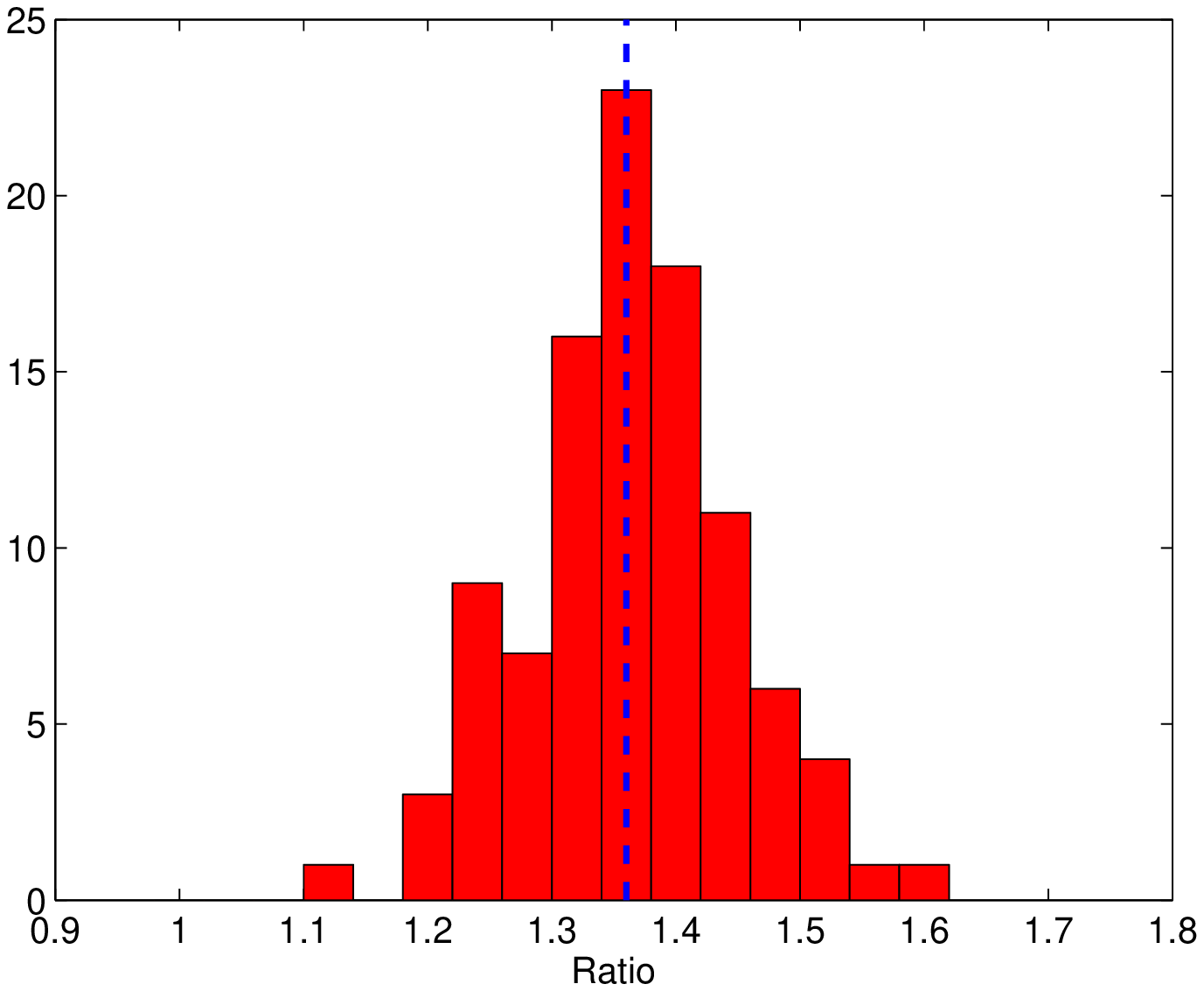}
      \caption{Histogram of the ratio between $trace(P_\infty)$ of deterministic schedule and $trace(EP_\infty)$ of stochastic schedule}
      \label{fig:fixed}
    \end{center}
  \end{minipage}
  \begin{minipage}[t]{0.4\textwidth}
    \begin{center}
      \includegraphics[width=0.9\textwidth]{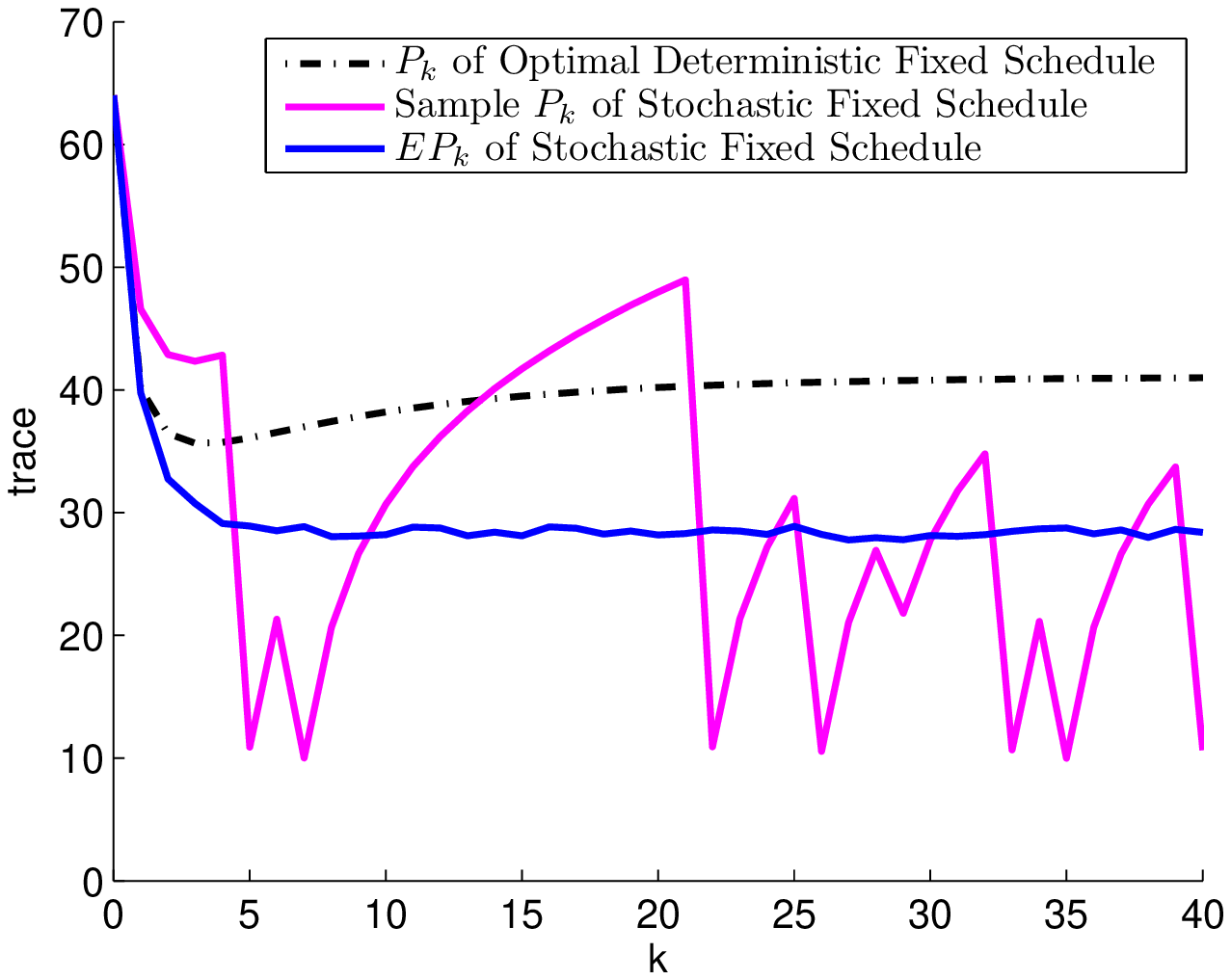}
      \caption{Evolution of $trace(P_k)$}
      \label{fig:ergodic}
    \end{center}
  \end{minipage}
\end{figure}

\section{Conclusions}\label{sec:conclusion}
In this paper, we propose a stochastic sensor selection algorithm for a tree topology wireless sensor network. We solve the optimal stochastic sensor selection problem after relaxation by means of convex optimization.
%
%
%
We also provide a possible implementation of our random sensor selection algorithm without introducing any communication overhead. Finally we discussed extensions to general graphs and to the case of unreliable communications. Examples show interesting results regarding the effectiveness of the proposed approach.

\bibliographystyle{IEEEtran}
\bibliography{SC_Letters09-reference}
\section*{Appendix}\label{sec:appendix}

\noindent First, let us state the following Proposition:
\begin{proposition}\label{proposition:convexity}
  Define functions $f(X),h(X)$ to be 
  \begin{align}
    f(X) &= X^{-1},\\
    h(X) &=(AX^{-1}A' + Q)^{-1}.
  \end{align}
where $X \in R^{n\times n}$  is positive definite and $T\in\mathcal T$. Then the following statements hold:
  \begin{enumerate}
    \item $f(X)$ is convex and monotone decreasing; 
    \item $h(X)$ is concave and monotone increasing; 
  \end{enumerate}
\end{proposition}

\begin{proof}[Proof of Theorem~\ref{theorem:lowerbound}]
	By the definition of $L_k$, we know that 
 \begin{align*}
   L^{-1}_k  = L^{-1}_{k|k-1} + \sum_{i=1}^m p_{k,i} \frac{C_i'C_i}{r_i},\, L^{-1}_{k|k-1} =(AL_{k-1}A' + Q)^{-1}.
 \end{align*}

 Let us define $Z_k \triangleq P_k^{-1},\,Z_{k|k-1}\triangleq P_{k|k-1}^{-1}$. We will first prove $L_k^{-1} \geq \mathbb EZ_k$ by induction. When $k=0$, $L_0^{-1} = Z_0 = P_0^{-1}$. Suppose that $L_{k-1}^{-1} \geq \mathbb EZ_{k-1}$, since $P_{k|k-1} = AP_{k-1}A'+Q$, we know that 
  \begin{equation}
      Z_{k|k-1} =(AZ_{k-1}^{-1}A' + Q)^{-1} = h(Z_{k-1}). 
  \end{equation}
  By taking the expectation on both sides, we get 
  \begin{equation}
    \mathbb EZ_{k|k-1} = \mathbb Eh(Z_{k-1}) \leq h(\mathbb EZ_{k-1}) \leq h(L^{-1}_{k-1}) = L^{-1}_{k|k-1}. 
  \end{equation}
  The first inequality is a consequence of the concavity of $h(X)$ and Jensen's inequality. The second inequality is derived from the monotonicity of $h(X)$ and from the fact that $L^{-1}_{k-1} \geq \mathbb EZ_{k-1}$. Now, by \eqref{eq:informationmatrix}, we know that 
  \begin{equation}
    \mathbb EZ_k = \mathbb EZ_{k|k-1} + \sum_{i=1}^m p_{k,i}\frac{C_i'C_i}{r_i} \leq  \mathbb EL^{-1}_{k|k-1} + \sum_{i=1}^m p_{k,i}\frac{C_i'C_i}{r_i} = L^{-1}_k. 
  \end{equation}
  Hence, for all $k$, $L^{-1}_k\geq \mathbb EZ_k$. Now, by the definition of $Z_k$, we know that
 \begin{displaymath}
   P_k = Z_k^{-1} = f(Z_k). 
 \end{displaymath}
 Since $f$ is convex, by Jensen's inequality, the following inequalities result 
 \begin{equation}
   \mathbb EP_k  = \mathbb Ef(Z_k) \geq f(\mathbb EZ_k) = (\mathbb EZ_k)^{-1} \geq L_k. 
 \end{equation}
\end{proof}

\begin{proof}[Proof of Theorem~\ref{theorem:lfunction}]
Fix $X$, 
\[
L(X,\mathbf p) = f\left((AXA^T+Q)^{-1} + \sum_{i=1}^m p_i \frac{C_iC_i'}{r_i}\right). 
\]
Since, $f$ is convex and $(AXA^T+Q)^{-1} + \sum_{i=1}^m p_i C_iC_i/r_i$ is linear with respect to $\mathbf p$, $L$ is convex with respect to $\mathbf p$. Once $\mathbf p$ is fixed, it is easy to see that $L$ is of the same form as $h$. By similar arguments, $L$ is concave and monotone decreasing with respect to $X$.
\end{proof}

Before proving Theorem~\ref{theorem:algorithm}, we need the following lemmas:

\begin{lemma}
  \label{lemma:existence1}
   Consider matrix $C_{\mathbf p} = [\sqrt{p_1}C_1',\ldots, \sqrt{p_m}C_m']'$. If the pair $(C_{\mathbf p}, A)$ is detectable, then the following limit exists for all positive semidefinite matrices $X$: 
  \begin{displaymath}
     L^{\infty}(X,\mathbf p) = \lim_{k\rightarrow\infty} L^{(k)}(X,\mathbf p). 
  \end{displaymath}
  Moreover, if the pair $(A,\,Q^{1/2})$ is controllable, then the above limit is unique regardless of $X$.
\end{lemma}

\begin{proof}
Let us build a linear system whose dynamics are given by 
  \begin{displaymath}
  \begin{array}{lcr}
    \tilde x_{k+1} & = & A \tilde x_k + \tilde w_k, \\
     \tilde y_k & = & C_{\mathbf p} \tilde x_k + \tilde v_k.
  \end{array}
  \end{displaymath}
where $\tilde x_0 \sim \mathcal N(0,X)$, $\tilde w_k \sim \mathcal N(0,Q)$, $\tilde v_k \sim \mathcal N(0,R)$ and all of them are mutually independent of each other. Consider now the covariance matrix of the Kalman filter for the above system, which is given by 
\begin{eqnarray}
  \tilde P_{0} &=& X, \\
   \tilde P_{k+1|k} &= & A\tilde P_{k}A^T + Q,\, \\
   \tilde P_{k+1} &=& (\tilde P_{k+1|k}^{-1} + c_{\mathbf{p}}R^{-1} c_{\mathbf{p}}')^{-1} = \left(\tilde P_{k+1|k}^{-1} + \sum_{i=1}^m p_i \frac{C_iC_i'}{r_i}\right)^{-1}.
  \label{eq:alterkalman} 
\end{eqnarray}
By construction, such a covariance matrix satisfies $\tilde P_{k} = L^{(k)}(X,p)$ and hence the limit $\tilde P_\infty = \lim_{k\rightarrow \infty}\tilde P_k$ exists if $(C_{\mathbf p}, A)$ is detectable. Moreover, the limit is unique regardless of $\tilde P_0$ if $(A,\,Q^{1/2})$ is controllable.
\end{proof}
\noindent Another theorem on the uniqueness of the limit can also be provided:

\begin{lemma}
  \label{lemma:existence2}
Let $Q > 0$ be a strictly positive definite matrix. If there exists a fixed point $X_0$ satisfying 
\begin{displaymath}
  X_0 = L (X_0 , \mathbf p), 
\end{displaymath}
then $L^\infty(X,\mathbf p)$ exists and moreover 
\begin{displaymath}
  L^\infty(X,\mathbf p) = X_0, \,\,\textrm{for all }X\textrm{ positive semidefinite}. 
\end{displaymath}
\end{lemma}

\begin{proof}
First, we want to show that $L(X,\mathbf p)$ is strictly positive for any $X \geq 0$. By definition we have 
\begin{displaymath}
   L (X , \mathbf p) = \left[(AXA^T+Q)^{-1} + \sum_{i=1}^m p_i \frac{C_iC_i'}{r_i}\right]^{-1} \geq  \left(Q^{-1} + \sum_{i=1}^m p_i \frac{C_iC_i'}{r_i}\right)^{-1}>0 . 
\end{displaymath}
In particular, this implies that $X_0 > 0$. Now, because $L(X,\mathbf p)$ is concave in $X$, we obtain: 
\begin{displaymath}
  \frac{1}{\alpha}L(\alpha X_0,\mathbf p)< \frac{1}{\alpha  }L(\alpha X_0,\mathbf p) + \frac{\alpha -1}{\alpha }L(0,\mathbf p) \leq L(X_0,\mathbf p) = X_0.\;\forall \alpha > 1 
\end{displaymath}
As a result, $L(\alpha X_0,\mathbf p) < \alpha X_0$ and, exploiting the monotonicity of $L(X,\mathbf p)$, the following inequality holds 
\begin{displaymath}
 0 < L^{(k+1)}(\alpha X_0,\mathbf p) < L^{(k)}(\alpha X_0,\mathbf p). 
\end{displaymath}
Then $L^{(k)}(\alpha X_0,\mathbf p)$ is bounded regardless of $k$. Because $X_0 > 0$ for any $X$ positive semidefinite, there exists a scalar $\alpha_x>1$, such that $X \leq \alpha _x X_0$, then, using again the monotonicity of $L(X,\mathbf p)$, one can prove that $ L^{(k)}(X,\mathbf p) < L^{(k)}(\alpha_x X_0,\mathbf p)$ is also bounded regardless of $k$. Hence, the pair $(C_{\mathbf p}, A)$ must be detectable, which implies that $L^\infty(X,\mathbf p)$ exists for all $X$. Moreover, since $Q > 0$, the limit is unique and it must be $X_0$.
\end{proof}

Now we are ready to prove Theorem~\ref{theorem:algorithm}
\begin{proof}
  \begin{enumerate}
    \item It is easy to check that $C_{\mathbf p_0} =\sqrt{ \mathcal E_d /(\sum_{i=1}^m c_i)}C$ and $\mathbf p_0\in \mathcal P$. Since $(C,\,A)$ is detectable, $(\sqrt{ \mathcal E_d /(\sum_{i=1}^m c_i)}C,A)$ is also detectable and then $L_0$ exists. 
    \item[5)] By the definition of accumulation point, there is a subsequence $\mathbf p_{i_1} ,\mathbf p_{i_2},\ldots$ which converges to $\mathbf p^*$. For each index $i_k$ we have 
  \begin{displaymath}
    L(L_{i_k-1},\mathbf p_{i_k}) = L_{i_k}. 
  \end{displaymath}
  If we take the limit on both side and exploit the fact that $L(X,\mathbf p)$ is continuous, we obtain 
   \begin{displaymath}
   L(L_\infty,\mathbf p^*) = L_\infty, 
  \end{displaymath}
  and finally by Lemma~\ref{lemma:existence2}, the limit is unique. 
  \end{enumerate}
\end{proof}

\end{document}